\documentclass[a4paper,reqno]{amsart}

\usepackage{amssymb}
\usepackage{latexsym}
\usepackage{amsmath}
\usepackage[mathcal]{euscript}
\usepackage{bbm}
\usepackage{tikz}
\usepackage{color}
\usepackage{pifont}

\def\cG{{\mathcal G}}      
      
      \def\cO{{\mathcal O}}

\def\cal H{{\mathcal H}}

\def\R{\mathbb{R}}
\def\C{\mathbb{C}}

\def\dom{{\text{\rm dom\,}}}

\def\phi{\varphi}

\DeclareMathOperator{\Res}{Res}

\DeclareMathOperator{\Imag}{Im}
\DeclareMathOperator{\spann}{span}
\DeclareMathOperator{\rank}{rank}

\newtheorem{theorem}{Theorem}[section]
\newtheorem*{thm*}{Theorem}
\newtheorem{proposition}[theorem]{Proposition}
\newtheorem{corollary}[theorem]{Corollary}
\newtheorem{lemma}[theorem]{Lemma}

\theoremstyle{definition}
\newtheorem{definition}[theorem]{Definition}
\newtheorem{example}[theorem]{Example}

\newtheorem{remark}[theorem]{Remark}
\newtheorem*{ack}{Acknowledgement}

\numberwithin{equation}{section}

\title{Visibility of quantum graph spectrum from the vertices}

\author[C.~K\"uhn]{Christian K\"uhn}
\address{TU Hamburg \\ Institut f\"ur Mathematik \\
Am Schwarzenberg-Campus~3 \\
Geb\"aude E \\
21073 Hamburg \\
Germany}
\email{christian.kuehn@tuhh.de}

\author[J.~Rohleder]{Jonathan Rohleder}
\address{Matematiska institutionen\\ Stockholms universitet \\
106 91 Stockholm \\
Sweden}
\email{jonathan.rohleder@math.su.se}

\begin{document}

\begin{abstract}
We investigate the relation between the eigenvalues of the Laplacian with Kirchhoff vertex conditions on a finite metric graph and a corresponding Titchmarsh--Weyl function (a parameter-dependent Neumann-to-Dirichlet map). We give a complete description of all real resonances, including multiplicities, in terms of the edge lengths and the connectivity of the graph, and apply it to characterize all eigenvalues which are visible for the Titchmarsh--Weyl function.
\end{abstract}

\maketitle

\section{Introduction}

Inverse problems for quantum graphs, i.e.\ differential operators on metric graphs, is an area which has attracted much interest during the last decade. Particular attention was paid to the problem of recovering a differential operator from the Titchmarsh--Weyl function on a selected set of vertices, see~\cite{AK08,B04,BW05,EK11,FY07,K09,K11,K13,R15,Y05,Y12}. The Titchmarsh--Weyl function is a parameter-dependent Dirichlet-to-Neumann or, equivalently, Neumann-to-Dirichlet map; it corresponds to ``boundary'' measurements and can be calculated from scattering data, see, e.g.,~\cite[Section~5.4]{BK13}. In contrast to the case of ordinary or partial differential operators the Titchmarsh--Weyl function for a quantum graph does not contain the full information on the differential operator on a metric graph $G$. This effect is caused by the existence of real resonances for the non-compact graph obtained from attaching infinite leads to the vertices of $G$; they materialize through eigenfunctions which vanish at all vertices of the graph. Their existence can also be understood as a lack of a unique continuation property for quantum graphs.

In this note we consider as a model operator the Laplacian $- \Delta_G$ on a finite metric graph $G$ subject to Kirchhoff (also sometimes called standard) matching conditions on all vertices; see Section~\ref{sec:preliminaries} for the details. Our aim is to study the relation between the (purely discrete) spectrum $\sigma (- \Delta_G)$ of $- \Delta_G$ and the poles of the Titchmarsh--Weyl function for a sufficiently large set of vertices. For any subset $B = \{v_1, \dots, v_m\}$ of the vertex set of $G$ the matrix-valued Titchmarsh--Weyl function can be defined via the relation 
\begin{align*}
 M_B (\mu) \begin{pmatrix} \partial_\nu f (v_1) \\ \vdots \\ \partial_\nu f (v_m) \end{pmatrix} = \begin{pmatrix} f (v_1) \\ \vdots \\ f(v_m) \end{pmatrix}, \quad \mu \in \C \setminus \sigma (- \Delta_G),
\end{align*}
where $f$ is any solution of $-f'' = \mu f$ on $G$ which is continuous on each vertex and satisfies the Kirchhoff condition on any vertex which does not belong to $B$; here $\partial_\nu f (v)$ is the sum of the derivatives at a vertex $v$ of the restrictions of $f$ to the edges incident to $v$. There is a close relation between the eigenvalues of $- \Delta_G$ and the poles of the meromorphic matrix function $\mu \mapsto M_B (\mu)$, see, e.g.,~\cite[Section~3.5.3]{BK13}. In fact, each pole of $M_B$ is an eigenvalue of $- \Delta_G$, but in general the converse is not true if cycles with commensurate edges are present, see, e.g.,~\cite{R15}.

In this note we focus on the case where the vertex set $B$ is rather large. We assume that $B$ contains all vertices of degree one as well as all {\em proper core vertices}; cf.~Definition~\ref{def:core}. In the main result Theorem~\ref{thm:visibility} of this note it turns out that exactly those eigenvalues $\lambda$ of $- \Delta_G$ appear as poles of $M_B$ for which 
\begin{align}\label{eq:condition}
 \dim \ker (- \Delta_G - \lambda) > \beta_1 (G_\lambda) - \beta_0^{\rm odd} (G_\lambda)
\end{align}
holds, where $G_\lambda$ is the subgraph of $G$ consisting of all edges whose lengths are natural multiples of $\pi / \sqrt{\lambda}$, $\beta_1 (G_\lambda)$ is the first Betti number (the number of independent cycles) of $G_\lambda$ and $\beta_0^{\rm odd} (G_\lambda)$ is the number of connected components of $G_\lambda$ which contain a cycle whose length is an odd multiple of $\pi / \sqrt{\lambda}$; cf.~Example~\ref{ex:example2} below. Moreover, an eigenvalue $\lambda$ is visible for the Titchmarsh--Weyl function with its full multiplicity if and only if $\beta_1 (G_\lambda) = \beta_0^{\rm odd} (G_\lambda)$, that is, each connected component of $G_\lambda$ contains at most one cycle and the length of this cycle (if it exists) is an odd multiple of $\pi / \sqrt{\lambda}$. In Example~\ref{ex:loop} we show that the condition~\eqref{eq:condition} can be violated, i.e., it may happen that eigenvalues of $- \Delta_G$ are totaly invisible for the Titchmarsh--Weyl function, even if $B$ contains all vertices of $G$.


To prove our main result, we trace the condition~\eqref{eq:condition} back to considerations on resonances of $G$. The core piece of our proof is Theorem~\ref{thm:general}, which states that a real point $\lambda > 0$ is a resonance of $G$ if and only if $\beta_1 (G_\lambda) > \beta_0^{\rm odd} (G_\lambda)$ holds, and that the dimension of the corresponding resonance eigenspace equals $\beta_1 (G_\lambda) - \beta_0^{\rm odd} (G_\lambda)$; this may be of independent interest. 

We would like to mention that a connection between the existence of resonance eigenfunctions (so-called scars) and cycles with commensurate edges was already observed in~\cite{SK03}. For further work related to resonances and scars for quantum graphs we refer the reader to~\cite{BKW04,C14,E14,EL10,GSS13,KS04,TM01,WS13} and the references therein. Moreover, we refer the reader to~\cite{BL16,F05} for genericity of simple eigenvalues and corresponding properties of edge lengths.

The present note is organized as follows. In Section~\ref{sec:preliminaries} we provide preliminaries on metric graphs, Laplacians and resonances. Section~\ref{sec:resonances} contains Theorem~\ref{thm:general} on the characterization of all resonances including multiplicities. Finally, in Section~\ref{sec:M} we discuss the relation to the Titchmarsh--Weyl function and provide our main result Theorem~\ref{thm:visibility}.

\begin{ack}
The authors gratefully acknowledge financial support by the Austrian Science Fund (FWF), project P~25162-N26. Moreover, the authors are grateful to Pavel Kurasov: a consideration on eigenfunctions of equilateral graphs in a preliminary version of his forthcoming book on quantum graphs inspired the present, simple form of the proof of Theorem~\ref{thm:general}. Finally, the authors would like to thank Johannes Cuno and Olaf Post for stimulating discussions and helpful remarks.
\end{ack}

\section{Preliminaries}\label{sec:preliminaries}

Let us shortly recall some fundamentals on Laplacians on metric graphs; for more details see the recent monographs~\cite{BK13,P12} or the survey~\cite{K08}; for the standard notions from graph theory, see, e.g.,~\cite{BM76}. A {\em finite metric graph} is a graph $G$, built from a finite set $V$ of vertices and a finite set $E$ of edges, together with a length function $L : E \to (0, \infty)$ such that each edge $e \in E$ is identified with the interval $[0, L (e)]$, where the endpoints $0$ and $L (e)$ are identified with the two vertices to which $e$ is incident. In the following all graphs are metric graphs and we just write $e = [0, L (e)]$ for any edge $e$. For an edge $e \in E$ and a vertex $v \in V$ we write $v = o (e)$ or $v = t (e)$ if the left or right endpoint of $e$, respectively, is identified with the vertex~$v$ and say that $e$ originates from or terminates at $v$, respectively. Note that if $e$ is a loop then $o (e) = t (e)$ holds. Let us recall, furthermore, that the {\em core} of a (discrete or metric) graph $G$ is the largest subgraph of $G$ which does not have any vertices of degree one.

Let $G$ be a finite metric graph. On $G$ we consider the Hilbert space $L^2 (G)$ of (equivalence classes of) square-integrable functions $f : G \to \C$, equipped with the standard norm and inner product. For any $f \in L^2 (G)$ and any edge $e \in E$ we denote by $f_e$ the restriction of $f$ to $e$, and we identify $f_e$ with a function on $[0, L (e)]$. Moreover, we write
\begin{align*}
 \widetilde H^k (G) = \left\{ f \in L^2 (G) : f_e \in H^k (e), e \in E \right\}, \quad k = 1, 2, \dots,
\end{align*}
where $H^k (e)$ is the Sobolev space of order $k$ on $(0, L (e))$. Note that any $f \in \widetilde H^k (G)$ is continuous on $(0, L (e))$ and $f |_{(0, L (e))}$ has a continuous extension to $[0, L (e)]$ for each $e \in E$. We say that a function $f \in \widetilde H^1 (G)$ is {\em continuous on $G$} if for each two edges $e$ and $\hat e$ which are incident to a joint vertex $v$ the limits of $f_e$ and $f_{\hat e}$ at~$v$ coincide. Accordingly we define
\begin{align*}
 H^1 (G) = \left\{ f \in \widetilde H^1 (G) : f~\text{is~continuous~on}~G \right\}.
\end{align*}
For $f \in H^1 (G)$ we write $f (v)$ for the limit of $f_e$ at $v$ for an arbitrary edge $e$ incident to $v$. Moreover, for $f \in H^1 (G) \cap \widetilde H^2 (G)$ we use the abbreviation
\begin{align*}
 \partial_\nu f (v) = \sum_{t (e) = v} f'_e (L (e)) - \sum_{o (e) = v} f'_e (0), \quad v \in V,
\end{align*}
for the sum of the derivatives pointing towards the vertex; here the evaluations of the derivatives must be understood as limits. 

The differential operator under consideration in this note is the Laplacian $- \Delta_G$ in~$L^2 (G)$ equipped with Kirchhoff (or standard) vertex conditions, i.e.\
\begin{align*}
 (- \Delta_G f)_e & = - f''_e, \quad e \in E, \\
 \dom (- \Delta_G) & = \left\{ f \in H^1 (G) \cap \widetilde H^2 (G) : \partial_\nu f (v) = 0, v \in V \right\}.
\end{align*}
The operator $- \Delta_G$ is selfadjoint in $L^2 (G)$ and its spectrum $\sigma (- \Delta_G)$ is nonnegative and consists of isolated eigenvalues with finite multiplicities; the smallest eigenvalue of $- \Delta_G$ equals zero and corresponds to those eigenfunctions which are constant on~$G$.

In the following the notion of a (real) resonance of $G$ plays an important role. We make the following definition.

\begin{definition}\label{def:resonance}
Let $- \Delta_G$ be the Laplacian in $L^2 (G)$ with Kirchhoff vertex conditions. A number $\lambda \in \R$ is called {\em (real) resonance of $G$} if there exists a nontrivial function $f \in \ker (- \Delta_G - \lambda)$ such that $f (v) = 0$ holds for each vertex $v$ of $G$. Each such $f$ is called {\em resonance eigenfunction}.
\end{definition}

Some remarks are in order. The classical definition of a resonance for (the Laplacian on) a metric graph refers to graphs which contain infinite leads, i.e.\ edges of infinite length attached to some vertices; see, e.g.,~\cite{DP11}. A resonance for such a non-compact graph is defined as a number $k \in \C \setminus \{0\}$ such that there exists a nontrivial solution $f$ of $- f'' = k^2 f$ on $G$ which satisfies the Kirchhoff condition at each vertex and is a scalar multiple of $e^{i k x}$ on each infinite lead. It can be seen easily that a resonance must satisfy $\Imag k \leq 0$, and it may happen that a resonance $k$ is real. If $G$ is a finite metric graph and $\widetilde G$ is the non-compact graph obtained from attaching infinite leads to all vertices of $G$ then it can be shown that $k$ is a real resonance of $\widetilde G$ in the classical sense if and only if $\lambda = k^2$ is a resonance of $G$ in the sense of Definition~\ref{def:resonance}. We work with $\lambda = k^2$ instead of $k$ since for our purposes this choice is more convenient. Note that also with respect to our definition a resonance of $G$ is nonzero (and, thus, positive) as the eigenfunctions of $- \Delta_G$ corresponding to $\lambda = 0$ are necessarily different from zero at each vertex.

\section{Resonances on the real line}\label{sec:resonances}

In this section we provide a characterization of all real resonances and their multiplicities of a finite metric graph $G$ in terms of its connectivity and algebraic properties of the edge lengths. In the statement of the following theorem, for a finite metric graph $G$ we denote its (first) Betti number by $\beta_1 (G)$, i.e., $\beta_1 (G) = |E| - |V| + \beta_0 (G)$, where $\beta_0 (G)$ is the number of connected components; recall that $\beta_1 (G)$ expresses the number of independent cycles in $G$. Moreover, for each $\lambda > 0$ we write
\begin{align*}
 R (G, \lambda) := \big\{ f \in \ker (- \Delta_G - \lambda) : f (v) = 0~\text{for all}~v \in V \big\}
\end{align*}
for the resonance eigenspace of $G$ corresponding to $\lambda$; note that $R (G, \lambda)$ is trivial for all but at most countably many $\lambda$. We say that a cycle in $G$ has {\em odd length} with respect to a given real number $x > 0$ if its length is an odd multiple of $x$.

\begin{theorem}\label{thm:general}
Let $G$ be a finite metric graph. For each $\lambda > 0$ denote by $G_\lambda$ the subgraph of $G$ consisting of all edges whose length is a natural multiple of $\pi / \sqrt{\lambda}$. Then
\begin{align}\label{eq:resSpace}
 \dim R (G, \lambda) = \beta_1 (G_\lambda) - \beta_0^{\rm odd} (G_\lambda),
\end{align}
where $\beta_0^{\rm odd} (G_\lambda)$ is the number of connected components of $G_\lambda$ which contain a cycle with odd length with respect to $\pi / \sqrt{\lambda}$. In particular, $\lambda$ is a resonance of $G$ if and only if the right-hand side of~\eqref{eq:resSpace} is nonzero.
\end{theorem}

We point out that the subgraph $G_\lambda$ defined in the theorem is empty for all but at most countably many $\lambda$. Moreover, we remark that if $G_\lambda$ is equilateral then $\beta_0^{\rm odd} (G_\lambda)$ is the number of connected components of $G_\lambda$ which are not bipartite. Finally, we point out that for equilateral metric graphs Theorem~\ref{thm:general} is in accordance with the results of~\cite[Section~3.1]{HW16}.

\begin{proof}[Proof of Theorem~\ref{thm:general}]
Let $\lambda \in \R$, let the subgraph $G_\lambda$ be defined as in the theorem, and let first $\cG_\lambda$ be an arbitrary connected component of $G_\lambda$. We denote by $\beta := \beta_1 (\cG_\lambda)$ its first Betti number and choose edges $e_1, \dots, e_\beta$ such that the graph $T$ obtained from $\cG_\lambda$ by removing $e_1, \dots, e_\beta$ is a connected tree. Then for each $j \in \{1, \dots, \beta\}$ the graph $T \cup \{e_j\}$ contains a cycle $C_j$, which is unique up to periodic shifts and reversion. 

Assume first that $\cG_\lambda$ contains no cycle with odd length w.r.t.~$\pi / \sqrt{\lambda}$. For each $j \in \{1, \dots, \beta\}$ consider a function $f^{(j)}$ on $G$ of the form
\begin{align*}
 f^{(j)}_e (x) = \alpha_e^{(j)} \sin (\sqrt{\lambda} x), \quad x \in [0, L (e)], \quad e \in E,
\end{align*}
with $\alpha_e^{(j)} = 0$ for all $e$ which do not belong to $C_j$. Since each edge of $C_j$ is a natural multiple of $\pi / \sqrt{\lambda}$, each such function $f^{(j)}$ is continuous on $G$ with $f^{(j)} (v) = 0$ for all $v \in V$, and since the total length of $C_j$ is an even multiple of $\pi / \sqrt{\lambda}$, we can choose $\alpha_e^{(j)} \in \{-1, 1\}$ for edges $e$ of $C_j$ in such a way that $f^{(j)}$ has a continuous derivative along $C_j$. Thus the $f^{(j)}$, $j = 1, \dots, \beta$, belong to $R (G, \lambda)$. Moreover, they are linearly independent since the restriction of $f^{(m)}$ to $e_j$ is identically zero for $m \neq j$ but nontrivial for $m = j$, $j = 1, \dots, \beta$. Furthermore, if $f \in R (G, \lambda)$ is arbitrary then $f_{e_j} = \gamma_j f^{(j)}_{e_j}$ holds for some $\gamma_j$, $j = 1, \dots, \beta$, and the function
\begin{align*}
 g := f - \sum_{j = 1}^\beta \gamma_j f^{(j)} 
\end{align*}
belongs to $R (G, \lambda)$ and vanishes identically on each of the edges $e_1, \dots, e_\beta$. Moreover, it is clear that $f$ (and, hence, $g$) must be zero identically on each edge of $G$ whose length is not a natural multiple of $\pi / \sqrt{\lambda}$. Therefore $g$ vanishes identically outside the tree $T$ and on all vertices of $G$, and thus it follows that $g = 0$ on $G$; cf.~Step~1 of the proof of~\cite[Theorem~3.5]{R15}. Hence in this case $\dim R (\cG_\lambda, \lambda) = \beta_1 (\cG_\lambda)$.

Let us now come to the case of a connected component $\cG_\lambda$ which contains a cycle with odd length w.r.t.~$\pi / \sqrt{\lambda}$. Then at least one of the cycles $C_j$ must have odd length w.r.t.~$\pi / \sqrt{\lambda}$, without loss of generality let this hold for $C_\beta$. Let us now go through the cycles $C_1, \dots, C_{\beta - 1}$. If $C_j$ has even length w.r.t.~$\pi / \sqrt{\lambda}$ then a nontrivial function $f^{(j)} \in R (G, \lambda)$ with support on $C_j$ can be found as in the previous case. If $C_j$ has odd length w.r.t.~$\pi / \sqrt{\lambda}$ then a nontrivial function $f^{(j)} \in R (G, \lambda)$ with support in $T \cup \{e_j, e_\beta\}$ can be constructed. Indeed, if $C_j$ and $C_\beta$ contain a joint edge then the symmetric difference of $C_j$ and $C_\beta$ is a cycle of even length and the construction of a resonance eigenfunction is as above; otherwise the core of $T \cup \{e_j, e_\beta\}$ consists of $C_j$, $C_\beta$ and, if these do not have a joint vertex, a path connecting the two, and the graph obtained from this by doubling the connecting path (if any) has an even multiple of $\pi / \sqrt{\lambda}$ as its length and contains an Eulerian cycle; thus ``spooling'' the function $x \mapsto \sin (\sqrt{\lambda} x)$ onto this Eulerian cycle and extending it by zero to all of $G$ leads to the desired resonance eigenfunction. Thus we have obtained linearly independent functions $f^{(1)}, \dots, f^{(\beta - 1)} \in R (G, \lambda)$. If $f \in R (G, \lambda)$ is arbitrary then as in the first case we can find coefficients $\gamma_1, \dots, \gamma_{\beta - 1}$ such that 
\begin{align*}
 g := f - \sum_{j = 1}^{\beta - 1} \gamma_j f^{(j)} 
\end{align*}
vanishes on each of the edges $e_1, \dots, e_{\beta - 1}$. Hence $g$ has support in $T \cup \{e_\beta\}$, and since $g$ vanishes at each vertex, it follows that actually the support of $g$ is contained in $C_\beta$, the only cycle of $T \cup \{e_\beta\}$. As $C_\beta$ has odd length w.r.t.~$\pi / \sqrt{\lambda}$, it follows $g = 0$ identically on $G$. Hence, in this case, $\dim R (\cG_\lambda, \lambda) = \beta_1 (\cG_\lambda) - 1$. 

Finally, denoting by $\cG_\lambda^1, \dots, \cG_\lambda^n$ the connected components of $G_\lambda$ containing no cycle with odd length w.r.t.~$\pi / \sqrt{\lambda}$ and by $\cG_\lambda^{n + 1}, \dots, \cG_\lambda^N$ the remaining ones, we obtain
\begin{align*}
 \dim R (G, \lambda) & = \sum_{k = 1}^N \dim R (\cG_\lambda^k, \lambda) = \sum_{k = 1}^n \beta_1 (\cG_\lambda^k) + \sum_{k = n + 1}^N \big( \beta_1 (\cG_\lambda^k) - 1 \big) \\
 & = \beta_1 (G_\lambda) - \beta_0^{\rm odd} (G_\lambda),
\end{align*}
where we have used that each resonance eigenfunction corresponding to $\lambda$ must have support in $G_\lambda$. We have derived the assertion of the theorem.
\end{proof}

We illustrate the application of Theorem~\ref{thm:general} by an example.

\begin{example}\label{ex:example2}
Let $G$ be the metric graph given in Figure~\ref{fig:Hantel1} with the dotted edge having length $\frac{\pi}{2}$, the dashed edges having length 1 and the remaining edges having length $\sqrt{3}$. 
\begin{figure}[h]
\begin{center}
\quad
\begin{tikzpicture}
\pgfsetlinewidth{1pt}
\color{black}
\fill (0,0) circle (2pt);  
\draw[dashed](0,0)--(0.866,1.5);
\draw[dashed](0,0)--(-0.866,1.5);
\draw[dashed](-0.866,1.5)--(0.866,1.5);
\fill (-0.866,1.5) circle (2pt);
\fill (0.866,1.5) circle (2pt);
\draw[dashed](0,0)--(0.866,-1.5);
\draw[dashed](0,0)--(-0.866,-1.5);
\draw[dashed](-0.866,-1.5)--(0.866,-1.5);
\fill (-0.866,-1.5) circle (2pt);
\fill (0.866,-1.5) circle (2pt);
\pgfxyline(0.866,1.5)(0.866,-1.5)
\pgfxyline(0.866,-1.5)(3.466,0)
\pgfxyline(0.866,1.5)(3.466,0)
\fill (3.466,0) circle (2pt);
\pgfxyline(-0.866,1.5)(-0.866,-1.5)
\pgfxyline(-0.866,-1.5)(-3.466,0)
\pgfxyline(-0.866,1.5)(-3.466,0)
\fill (-3.466,0) circle (2pt);
\draw[dotted](3.466,0)--(6.5,0);
\fill (6.5,0) circle (2pt);
\end{tikzpicture}
\qquad
\end{center}
\caption{A metric graph consisting of a dotted edge having length $\frac{\pi}{2}$, six dashed edges having length 1 and the remaining edges having length $\sqrt{3}$.}\label{fig:Hantel1}
\end{figure}
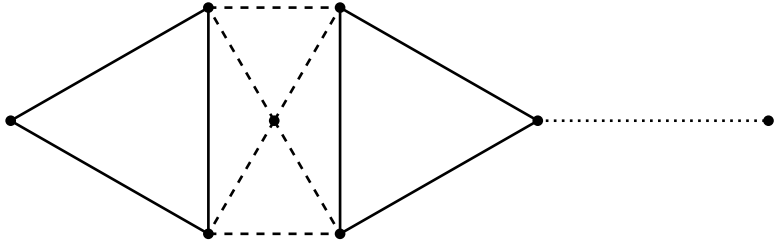
In the table below the first four values of $\lambda$ are displayed for which the subgraph $G_\lambda$ in Theorem~\ref{thm:general} is nontrivial, and the dimension of $R (G, \lambda)$ is calculated for these $\lambda$ by means of Theorem~\ref{thm:general}. It turns out that two out of these four $\lambda$ are indeed resonances, i.e., $\dim R (G, \lambda) > 0$.
\vspace*{2mm}
\begin{center}
\begin{tabular}{|c|c|c|c|c|c|c|}
\hline
$\lambda$ \phantom{$\hat{\hat I}$ \hspace{-4mm}} & $\frac{\pi}{\sqrt{\lambda}}$ & $G_\lambda$ & $\beta_1(G_\lambda)$ & $\beta_0^\text{odd}(G_\lambda)$ & $\dim R(G,\lambda)$ & resonance? \\ \hline
&&&&&&\\[-0.3cm]
$\frac{1}{3}\pi^2$	&$\sqrt{3}$	              &\resizebox{2cm}{0.5cm}{%
\begin{tikzpicture}
\pgfsetlinewidth{3pt}
\color{lightgray}
\fill (0,0) circle (2pt);  
\pgfxyline(0,0)(0.866,1.5)
\pgfxyline(0,0)(-0.866,1.5)
\pgfxyline(-0.866,1.5)(0.866,1.5)
\fill (-0.866,1.5) circle (2pt);
\fill (0.866,1.5) circle (2pt);
\pgfxyline(0,0)(0.866,-1.5)
\pgfxyline(0,0)(-0.866,-1.5)
\pgfxyline(-0.866,-1.5)(0.866,-1.5)
\fill (-0.866,-1.5) circle (2pt);
\fill (0.866,-1.5) circle (2pt);
\color{black}
\pgfxyline(0.866,1.5)(0.866,-1.5)
\pgfxyline(0.866,-1.5)(3.466,0)
\pgfxyline(0.866,1.5)(3.466,0)
\fill (3.466,0) circle (2pt);
\pgfxyline(-0.866,1.5)(-0.866,-1.5)
\pgfxyline(-0.866,-1.5)(-3.466,0)
\pgfxyline(-0.866,1.5)(-3.466,0)
\fill (-3.466,0) circle (2pt);
\color{lightgray}
\pgfxyline(3.466,0)(6.5,0)
\fill (6.5,0) circle (2pt);
\end{tikzpicture}
}
&2 &2 &0 & \ding{55} \\[0.2cm]
$4$	&$\frac{\pi}{2}$	&
\resizebox{2cm}{0.5cm}{%
\begin{tikzpicture}
\pgfsetlinewidth{3pt}
\color{lightgray}
\fill (0,0) circle (2pt);  
\pgfxyline(0,0)(0.866,1.5)
\pgfxyline(0,0)(-0.866,1.5)
\pgfxyline(-0.866,1.5)(0.866,1.5)
\fill (-0.866,1.5) circle (2pt);
\fill (0.866,1.5) circle (2pt);
\pgfxyline(0,0)(0.866,-1.5)
\pgfxyline(0,0)(-0.866,-1.5)
\pgfxyline(-0.866,-1.5)(0.866,-1.5)
\fill (-0.866,-1.5) circle (2pt);
\fill (0.866,-1.5) circle (2pt);
\pgfxyline(0.866,1.5)(0.866,-1.5)
\pgfxyline(0.866,-1.5)(3.466,0)
\pgfxyline(0.866,1.5)(3.466,0)
\fill (3.466,0) circle (2pt);
\pgfxyline(-0.866,1.5)(-0.866,-1.5)
\pgfxyline(-0.866,-1.5)(-3.466,0)
\pgfxyline(-0.866,1.5)(-3.466,0)
\fill (-3.466,0) circle (2pt);
\color{black}
\pgfxyline(3.466,0)(6.5,0)
\fill (6.5,0) circle (2pt);
\end{tikzpicture}
}
&0 &0 &0 & \ding{55} \\[0.2cm]
$\pi^2$	&$1$  &
\resizebox{2cm}{0.5cm}{%
\begin{tikzpicture}
\pgfsetlinewidth{3pt}
\color{black}
\fill (0,0) circle (2pt);  
\pgfxyline(0,0)(0.866,1.5)
\pgfxyline(0,0)(-0.866,1.5)
\pgfxyline(-0.866,1.5)(0.866,1.5)
\fill (-0.866,1.5) circle (2pt);
\fill (0.866,1.5) circle (2pt);
\pgfxyline(0,0)(0.866,-1.5)
\pgfxyline(0,0)(-0.866,-1.5)
\pgfxyline(-0.866,-1.5)(0.866,-1.5)
\fill (-0.866,-1.5) circle (2pt);
\fill (0.866,-1.5) circle (2pt);
\color{lightgray}
\pgfxyline(0.866,1.5)(0.866,-1.5)
\pgfxyline(0.866,-1.5)(3.466,0)
\pgfxyline(0.866,1.5)(3.466,0)
\fill (3.466,0) circle (2pt);
\pgfxyline(-0.866,1.5)(-0.866,-1.5)
\pgfxyline(-0.866,-1.5)(-3.466,0)
\pgfxyline(-0.866,1.5)(-3.466,0)
\fill (-3.466,0) circle (2pt);
\pgfxyline(3.466,0)(6.5,0)
\fill (6.5,0) circle (2pt);
\end{tikzpicture}
}
&2 &1 &1 & \ding{51} \\[0.2cm]
$\frac{4}{3}\pi^2$	&$\frac{\sqrt{3}}{2}$	          &
\resizebox{2cm}{0.5cm}{%
\begin{tikzpicture}
\pgfsetlinewidth{3pt}
\color{lightgray}
\fill (0,0) circle (2pt);  
\pgfxyline(0,0)(0.866,1.5)
\pgfxyline(0,0)(-0.866,1.5)
\pgfxyline(-0.866,1.5)(0.866,1.5)
\fill (-0.866,1.5) circle (2pt);
\fill (0.866,1.5) circle (2pt);
\pgfxyline(0,0)(0.866,-1.5)
\pgfxyline(0,0)(-0.866,-1.5)
\pgfxyline(-0.866,-1.5)(0.866,-1.5)
\fill (-0.866,-1.5) circle (2pt);
\fill (0.866,-1.5) circle (2pt);
\color{black}
\pgfxyline(0.866,1.5)(0.866,-1.5)
\pgfxyline(0.866,-1.5)(3.466,0)
\pgfxyline(0.866,1.5)(3.466,0)
\fill (3.466,0) circle (2pt);
\pgfxyline(-0.866,1.5)(-0.866,-1.5)
\pgfxyline(-0.866,-1.5)(-3.466,0)
\pgfxyline(-0.866,1.5)(-3.466,0)
\fill (-3.466,0) circle (2pt);
\color{lightgray}
\pgfxyline(3.466,0)(6.5,0)
\fill (6.5,0) circle (2pt);
\end{tikzpicture}
}
&2 &0 &2 & \ding{51} \\[0.2cm]
\hline
\end{tabular}
\end{center}
\end{example}
\bigskip \bigskip

An immediate corollary of Theorem~\ref{thm:general} looks as follows. Here for a given cycle $C$ in $G$ with commensurate edges (i.e., the edge lengths are rational multiples of each other) we denote by $u_C$ the maximal number such that the length of each edge in $C$ is a natural multiple of $u_C$.

\begin{corollary}\label{cor:soGehtsAuch}
Let $G$ be a finite metric graph and let
\begin{align*}
 \lambda_G := \inf \left\{ \frac{\pi^2}{u_C^2} : C~\text{cycle in}~G~\text{with commensurate edges} \right\}.
\end{align*}
Then $G$ has no resonance in $(- \infty, \lambda_G)$.
\end{corollary}

\begin{proof}
Let $\lambda < \lambda_G$. Then $\pi / \sqrt{\lambda} > u_C$ holds for each cycle $C$ in $G$ with commensurate edges; in particular, not every edge length in $C$ is a natural multiple of $\pi / \sqrt{\lambda}$. Thus no cycle with commensurate edges is contained in $G_\lambda$. Since $G_\lambda$ also does not contain any non-commensurate cycles, it follows that $G_\lambda$ is a tree and Theorem~\ref{thm:general} yields $\dim R (G, \lambda) = 0$, that is, $\lambda$ is not a resonance of $G$.
\end{proof}

The expression $\lambda_G$ in Corollary~\ref{cor:soGehtsAuch} is comparatively simple and we point out that in some cases $\lambda_G$ is optimal in the sense that $\lambda_G$ is indeed a (the smallest) real resonance of $G$, while for other graphs this is not the case. However, Theorem~\ref{thm:general} can always be used to find the actual lowest real resonance.

\section{Visibility of quantum graph spectrum from the vertices}\label{sec:M}

In this section we apply the result of Section~\ref{sec:resonances} to the question of visibility of quantum graph spectrum from vertex data which is given via the Titchmarsh--Weyl function.

We start by recalling the definition of the Titchmarsh--Weyl function. For its well-definedness see, e.g.,~\cite[Lemma~2.2]{R15}.

\begin{definition}\label{def:TW}
Let $B = \{v_1, \dots, v_m\}$ be a nonempty subset of the set of vertices of~$G$. Moreover, let $\mu \in \C \setminus \sigma (- \Delta_G)$. For $l = 1, \dots, m$ let $f^{(l)} \in \widetilde H^2 (G) \cap H^1 (G)$ with $- f_e^{(l) \prime \prime} = \mu f_e^{(l)}$, $e \in E$, such that $\partial_\nu f^{(l)} (v_l) = 1$ and
\begin{align*}
  \partial_\nu f^{(l)} (v) = 0, \quad v \in V \setminus \{v_l\}.
\end{align*}
The {\em Titchmarsh--Weyl matrix} $M_B (\mu) \in \C^{m \times m}$ is the matrix with the entries
\begin{align*}
 (M_B (\mu))_{k, l} = f^{(l)} (v_k), \quad k, l = 1, \dots, m.
\end{align*}
The matrix function $\mu \mapsto M_B (\mu)$ is called {\em Titchmarsh--Weyl function}.
\end{definition}

Obviously the Titchmarsh--Weyl function depends on the chosen vertex set $B$. Below we will fix $B$ sufficiently large for our purposes. Note that the Titchmarsh--Weyl function satisfies
\begin{align*}
 M_B (\mu) \begin{pmatrix} \partial_\nu f (v_1) \\ \vdots \\ \partial_\nu f (v_m) \end{pmatrix} = \begin{pmatrix} f (v_1) \\ \vdots \\ f (v_m) \end{pmatrix}, \quad \mu \in \C \setminus \sigma (- \Delta_G),
\end{align*}
for any solution $f \in \widetilde H^2 (G) \cap H^1 (G)$ of $- f_e'' = \mu f_e$, $e \in E$, such that $\partial_\nu f (v) = 0$ for all $v \in V \setminus B$. For fixed $\mu \in \C \setminus \sigma (- \Delta_G)$ the matrix $M_B (\mu) \in \C^{m \times m}$ can be regarded as a Neumann-to-Dirichlet map for the graph Laplacian. 

In the following for an analytic matrix function $R$ defined on an open set $\Omega \subset \C$ we say that $\lambda \in \C$ is a pole of order $n$ of $R$ if there exists an open neighborhood $\cO$ of $\lambda$ in $\C$ such that $\cO \setminus \{\lambda\} \subset \Omega$,
\begin{align*}
 \lim_{\mu \to \lambda} (\mu - \lambda)^n R (\mu)~\text{exists and is nontrivial}, \quad \text{and} \quad \lim_{\mu \to \lambda} (\mu - \lambda)^{n + 1} R (\mu) = 0.
\end{align*}
We say that $R$ is meromorphic if $\C \setminus \Omega$ consists of isolated points which are poles or removable singularities. For any $\lambda \in \C$ we denote by $\Res_\lambda R$ the residue of $R$ at $\lambda$, i.e., the matrix given by
\begin{align*}
 \Res_\lambda R = \frac{1}{2 \pi i} \int_\Gamma R (\mu) d \mu,
\end{align*}
where $\Gamma$ is any closed Jordan curve in $\Omega$ which surrounds $\lambda$ but no other point in $\C \setminus \Omega$. In particular, if $\lambda \in \Omega$ or if $\lambda$ is a removable singularity of $R$ then $\Res_\lambda R = 0$. Moreover, $\lambda \in \C \setminus \Omega$ is a pole of $R$ if and only if $\rank \Res_\lambda R > 0$.

For the following proposition see~\cite[Proposition~2.4]{R15}. Its proof is based on an expression for the Titchmarsh--Weyl function in terms of the resolvent of the Kirchhoff Laplacian.

\begin{proposition}\label{prop:M}
Let $B = \{v_1, \dots, v_m\}$ be any nonempty set of vertices of $G$. Then the matrix-valued Titchmarsh--Weyl function $\mu \mapsto M_B (\mu)$ in Definition~\ref{def:TW} is meromorphic on $\C$ with possible poles of order one in the discrete set $\sigma (- \Delta_G)$. Moreover, for each $\lambda \in \R$ the identity 
\begin{align*}
 \dim \ker (- \Delta_G - \lambda) = \rank \Res_\lambda M_B + \dim R (G, \lambda)
\end{align*}
holds, where
\begin{align*}
 R (G, \lambda) := \big\{ f \in \ker ( - \Delta_G - \lambda) : f (v) = 0~\text{for all}~v \in B \big\}.
\end{align*}
\end{proposition}

It follows immediately from Proposition~\ref{prop:M} that each pole of $M_B$ is an eigenvalue of $- \Delta_G$. However, eigenvalues may occur which do not materialize as poles of $M_B$. 

\begin{remark}
There is an alternative way of describing the eigenvalues of the Kirchhoff Laplacian in terms of the scattering matrix $S (k)$ for the non-compact graph obtained from $G$ by attaching infinite leads to the vertices in $B$. In fact, the points $\lambda = k^2$ such that $\det (I - S (k)) = 0$ correspond exactly to those eigenvalues having non-resonant eigenfunctions, whereas the resonances are not directly visible from the scattering matrix. This observation can be transferred to the poles of the Neumann-to-Dirichlet map by expressing $M_B (\lambda)$ in terms of the scattering matrix (via a Cayley transform). A description of the complete spectrum of the Laplacian with Kirchhoff (and further) vertex conditions including the resonances in terms of scattering data was given in~\cite[Section~2]{BBS12}, see also~\cite[Theorem~5.4.6]{BK13}.
\end{remark}

For a choice of the vertex set $B$ which is suitable to apply Theorem~\ref{thm:general}, let us recall the following definition; cf.~Figure~\ref{fig:core}. Recall that the core of a graph is the largest subgraph which does not possess vertices of degree one.

\begin{definition}\label{def:core}
Let~$G$ be a finite metric graph. 
\begin{enumerate}
 \item A {\em boundary vertex} is a vertex of degree one.
 \item We call a vertex~$v$ of~$G$ {\em proper core vertex} if all edges attached to~$v$ belong to the core of~$G$.
\end{enumerate}
\end{definition}

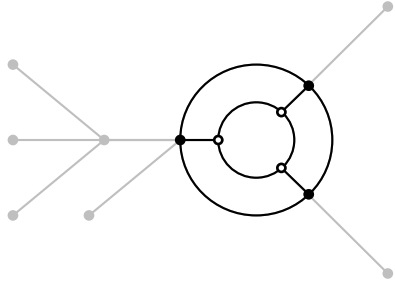
\begin{figure}[h]
\begin{center}
\begin{tikzpicture}
\pgfsetlinewidth{0.8pt}
\color{lightgray}
\draw (4.5,0) circle (5mm);
\draw (4.5,0) circle (10mm);
\pgfxyline(1.3,0)(4,0)
\pgfxyline(4.83,0.37)(6.23,1.77)
\pgfxyline(4.83,-0.37)(6.23,-1.77)
\fill (4,0) circle (2pt);
\fill (4.83,0.37) circle (2pt);
\fill (4.83,-0.37) circle (2pt);
\fill (1.3,0) circle (2pt);
\fill (6.23,1.77) circle (2pt);
\fill (6.23,-1.77) circle (2pt);
\fill (5.19,0.72) circle (2pt);
\fill (5.19,-0.72) circle (2pt);
\fill (3.5,0) circle (2pt);
\pgfxyline(2.5,0)(1.3,1);
\pgfxyline(2.5,0)(1.3,-1);
\pgfxyline(3.5,0)(2.3,-1);
\fill (2.5,0) circle (2pt);
\fill (1.3,1) circle (2pt);
\fill (1.3,-1) circle (2pt);
\fill (2.3,-1) circle (2pt);
\pgfxyline(2.5,0)(4,0)
\color{black}
\draw (4.5,0) circle (5mm);
\draw (4.5,0) circle (10mm);
\pgfxyline(3.5,0)(4,0)
\pgfxyline(4.83,0.37)(5.19,0.72)
\pgfxyline(4.83,-0.37)(5.19,-0.72)
\fill (4,0) circle (2pt);
\fill (4.83,0.37) circle (2pt);
\fill (4.83,-0.37) circle (2pt);
\fill (5.19,0.72) circle (2pt);
\fill (5.19,-0.72) circle (2pt);
\fill (3.5,0) circle (2pt);
\color{white}
\fill (4,0) circle (1pt);
\fill (4.83,0.37) circle (1pt);
\fill (4.83,-0.37) circle (1pt);
\pgfputat{\pgfxy(3.4,0.73)}{\pgfbox[center,base]{$C$}}
\end{tikzpicture}
\end{center}
\caption{A graph with its core marked in black and the proper core vertices drawn with empty dots.}
\label{fig:core}
\end{figure}

We are going to consider the Titchmarsh--Weyl function for any vertex set $B$ which contains all boundary vertices and all proper core vertices of $G$. As a preparation we show the following simple lemma. 

\begin{lemma}\label{lem:billig}
Let $B$ be a set of vertices which contains all boundary vertices and all proper core vertices of $G$. Let $\lambda \in \R$ and let $f \in \ker (- \Delta_G - \lambda)$ such that $f (v) = 0$ for all $v \in B$. Then $f (v) = 0$ for all $v \in V$.
\end{lemma}

\begin{proof}
The graph $G$ consists of its core and a finite number of rooted trees with their roots belonging to the core of $G$, and if $T$ is any of these trees, by assumption, $f (v) = 0$ for all but at most one boundary vertex $v$ of $T$. Thus a reasoning as in Step~1 of the proof of~\cite[Theorem~3.5]{R15} yields $f = 0$ identically on $T$, in particular, $f (v) = 0$ for each vertex $v$ of $T$. Since $f (v) = 0$ for each proper core vertex $v$ of $G$, it follows $f (v) = 0$ for all $v \in V$.
\end{proof}

We are now able to formulate the main result of this section, which is a direct consequence of Theorem~\ref{thm:general} and Proposition~\ref{prop:M}.

\begin{theorem}\label{thm:visibility}
Let $G$ be a finite metric graph and let $B \subset V$ be a set of vertices which contains all boundary vertices and all proper core vertices of $G$. For each $\lambda > 0$ let $G_\lambda$ be the subgraph of $G$ consisting of all edges whose lengths are natural multiples of $\pi / \sqrt{\lambda}$ and let $\beta_0^{\rm odd} (G_\lambda)$ be the number of connected components of $G_\lambda$ which contain a cycle with odd length with respect to $\pi / \sqrt{\lambda}$. Then the following assertions hold.
\begin{enumerate}
 \item If $\beta_1 (G_\lambda) = \beta_0^{\rm odd} (G_\lambda)$ then $\lambda$ is an eigenvalue of $- \Delta_G$ if and only if $\lambda$ is a pole of~$M_B$. Moreover,
\begin{align*}\
 \dim \ker (- \Delta_G - \lambda) = \rank \Res_\lambda M_B
\end{align*}
holds.
 \item If $\beta_1 (G_\lambda) > \beta_0^{\rm odd} (G_\lambda)$ then $\lambda$ is an eigenvalue of $- \Delta_G$ and 
\begin{align*}
 \dim \ker (- \Delta_G - \lambda) = \rank \Res_\lambda M_B + \beta_1 (G_\lambda) - \beta_0^{\rm odd} (G_\lambda).
\end{align*}
In particular, $\lambda$ is a pole of $M_B$ if and only if
\begin{align*}
 \dim \ker (- \Delta_G - \lambda) > \beta_1 (G_\lambda) - \beta_0^{\rm odd} (G_\lambda)
\end{align*}
holds.
\end{enumerate}
\end{theorem}

Note that the theorem covers all possible cases since $\beta_1 (G_\lambda) \geq \beta_0^{\rm odd} (G_\lambda)$ is always true. Moreover, note that the condition $\beta_1 (G_\lambda) = \beta_0^{\rm odd} (G_\lambda)$ in (i) means that each connected component of $G_\lambda$ may at most contain one cycle and, if so, this cycle needs to have odd length w.r.t.~$\pi / \sqrt{\lambda}$.

The following example shows that it may in fact happen that an eigenvalue $\lambda$ of $- \Delta_G$ is invisible for $M_B$, that is, $\lambda$ is a removable singularity of $M_B$, even for $B = V$.

\begin{example}\label{ex:loop}
Consider the metric graph in Figure~\ref{fig:loop} which consists of a loop and a further edge attached to it. We assume $L (e_1) = \pi$ and $L (e_2) = 1$. 
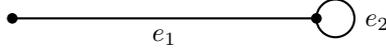
\begin{figure}[h]
\begin{center}
\begin{tikzpicture}
\pgfsetlinewidth{0.8pt}
\color{black}
\draw (4.25,0) circle (2.5mm);
\fill (4,0) circle (2pt);
\fill (0,0) circle (2pt);
\pgfxyline(0,0)(4,0)
\pgfputat{\pgfxy(2.0,-0.3)}{\pgfbox[center,base]{$e_1$}}
\pgfputat{\pgfxy(4.8,-0.1)}{\pgfbox[center,base]{$e_2$}}
\end{tikzpicture}
\end{center}
\caption{The metric graph~$G$ in Example~\ref{ex:loop}.}
\label{fig:loop}
\end{figure}
Then an application of Theorem~\ref{thm:general} yields that the smallest real resonance of $G$ is given by
\begin{align*}
 \lambda = 4 \pi^2.
\end{align*}
Furthermore, a simple calculation yields that the corresponding eigenspace of $- \Delta_G$ is one-dimensional and is spanned by the function $f$ with
\begin{align*}
 f_{e_1} (x) = 0, \quad x \in e_1, \qquad \text{and} \qquad f_{e_2} (x) = \sin (2 \pi x), \quad x \in e_2.
\end{align*}
In particular, we have 
\begin{align*}
 R (G, \lambda) = \spann \{f\} = \ker (- \Delta_G - \lambda)
\end{align*}
and Proposition~\ref{prop:M} yields $\rank \Res_\lambda M_B = 0$ for any (nonempty) choice of $B$. Hence the eigenvalue $\lambda$ is a removable singularity of $M_B$ or, in other words, it is invisible for~$M_B$.
\end{example}

We close this note with noting a consequence of Theorem~\ref{thm:visibility} and Corollary~\ref{cor:soGehtsAuch}.

\begin{corollary}
Let $G$ be a finite metric graph and let $B \subset V$ be a set of vertices which contains all boundary vertices and all proper core vertices of $G$. Define $\lambda_G$ as in Corollary~\ref{cor:soGehtsAuch} and let $\lambda < \lambda_G$. Then $\lambda$ is an eigenvalue of $- \Delta_G$ if and only if $\lambda$ is a pole of~$M_B$. Moreover,
\begin{align*}\
 \dim \ker (- \Delta_G - \lambda) = \rank \Res_\lambda M_B
\end{align*}
holds.
\end{corollary}

\end{document}